\documentclass[11pt]{article}
\usepackage{latexsym, amsmath, amsfonts, amsthm, amssymb,bbm,calrsfs}
\usepackage{a4wide}
\usepackage[shortlabels]{enumitem}
\usepackage[colorlinks=true]{hyperref}

\newcommand{\R}{\mathbb R}
\newcommand{\E}{\mathbb E}

\newcommand{\e}{\mathrm e}

\renewcommand{\phi}{\varphi}

\renewcommand{\hat}{\widehat}

\newtheorem{theorem}{Theorem}

\newtheorem{lemma}[theorem]{Lemma}

\theoremstyle{remark}

\theoremstyle{definition}

\title{On the stochastic proof
of the Blaschke-Santal\'o inequality}
\author{Joseph Lehec} 

\begin{document}

\maketitle

\begin{abstract}
In 2024, Courtade, Fathi and Mikulincer gave a proof 
of the symmetrized Talagrand inequality based 
on stochastic calculus, in the spirit of Borell's proof 
of the Prékopa-Leindler inequality. The symmetrized 
Talagrand inequality can be seen as a dual form of the 
functional Santal\'o inequality. The modest 
purpose of this note is to give a simplified version 
of the Courtade, Fathi and Mikulincer argument. Namely 
we first recall briefly Borell's original argument,
and we then explain a simple twist in his proof that allows 
to recover the functional Santal\'o inequality 
directly, rather than in its dual form. 
\end{abstract}

\section{Introduction} 
Given a convex body $K$ in $\R^n$, namely a compact convex set with non empty interior, we let $K^\circ = \{ y\in \R^n \colon x\cdot y\leq 1 ,\; \forall x\in K\}$ be its polar body. The volume product of $K$ is the quantity $\vert K \vert \cdot \vert K^\circ \vert$ where $\vert \cdot \vert$ denotes the Lebesgue measure. 
The Blaschke-Santal\'o inequality asserts the among convex bodies having their barycenter at $0$, the volume product is maximal when $K$ is the Euclidean ball. This was proved by Blaschke in dimension $3$ and Santal\'o~\cite{santalo} in the general case, 
using variational methods. More recent proofs~\cite{S-R,MP} rely on symmetrization arguments. In this paper we shall discuss the following functional 
version of the Blaschke-Santal\'o inequality. 
\begin{theorem}[Functional Santal\'o inequality] \label{thm_sant}
Let $f,g$ be non-negative measurable functions on $\R^n$
satisfying
\begin{equation} \label{eq_sant} 
f(x)g(y)\leq \e^{-x\cdot y} , \quad \forall x,y\in\R^n .
\end{equation}
If $f$ (or $g$) has its barycenter at $0$, namely if $\int_{\R^n} x f(x) \, dx = 0$,
then 
\[ 
\int_{\R^n} f(x) \, dx \int_{\R^n} g(y) \, dy \leq  (2\pi)^n . 
\]   
\end{theorem} 
Applying this to $f = \exp ( - \Vert \cdot \Vert^2_K /2) $ where $\Vert \cdot \Vert_K$ 
is the gauge function of $K$ recovers the geometric version.  
The functional inequality was first proved by Ball~\cite{ball} under the additional assumption that $f$ is even. Theorem~\ref{thm_sant} is essentially due to Artstein, Klartag and Milman~\cite{AKM}, with the caveat that in their paper the function having its barycenter at $0$ should be log-concave. These works relied on the geometric version of 
the inequality, by applying to level sets of the function $f$. The first direct proof, based on some induction on the dimension, is due the author in~\cite{lehecCRAS}. See also~\cite{lehecADM} where a more general inequality, originally due to Fradelizi and Meyer~\cite{FM}, is established. 

Recently, Nakamura and Tsuji~\cite{NT} gave a new proof of the 
functional Santal\'o inequality, based on a semigroup argument. 
Namely, they proved that when $f$ is even
the expression $\int_{\R^n} f^\circ(y) \, dy$ is monotone increasing along 
the Fokker-Planck semigroup. Here $f^\circ$ denotes the polar function of $f$, 
namely the largest function $g$ such that~\eqref{eq_sant} is satisfied. This 
recovers Santal\'o (for even functions) by sending time to $+\infty$. 
See also~\cite{CGNT}, in which the argument is somewhat simplified.  

In 2024, Courtade, Fathi and Mikulincer~\cite{CFM} found a
proof based on stochastic calculus 
of a dual version of the functional Santal\'o inequality
related to Talagrand's tranport inequality. 
Stochastic proofs of functional inequalities 
seem to have started with Borell's seminal paper~\cite{borell} on the 
Pr\'ekopa-Leindler inequality, which is a functional counterpart of 
the Brunn-Minkowski inequality. 
The modest purpose of this note is to give a slightly 
simplified version of the Courtade, Fathi, and Mikulincer
argument in which the connection with Borell's work is more 
apparent.
\section{Borell's formula} 
Throughout this section we are given a standard Brownian 
motion on $\R^n$, denoted $(B_t)_{t\in[0,1]}$ (all our processes 
will be indexed by the finite time interval $[0,1]$). 
In this context we call \emph{drift} 
a progressively measurable process $(u_t)$ such that 
\[ 
\E \int_0^1 \vert u_t\vert^2 \,dt < + \infty. 
\] 
The standard Gaussian measure is denoted $\gamma_n$. 
Borell's proof of the Pr\'ekopa-Leindler 
inequality relies on the following representation formula. 
\begin{lemma}[Borell's formula]
Let $\phi \colon \R^n \to \R$ be measurable and bounded from above. 
Then 
\[ 
\log \left( \int_{\R^n} \e^\phi \, d\gamma_n \right) 
= \sup \left\{ \E \left[\phi \left( B_1 + \int_0^1  u_t \, dt \right) 
 -  \frac12 \int_0^1 \vert u_t\vert^2 \, dt \right] \right\} 
 \] 
where the supremum is taken over all drifts $(u_t)$. 
\end{lemma} 
In~\cite{borell} Borell proves this formula and gives 
a nice application to the Pr\'ekopa-Leindler 
inequality, to which we shall come back below. A dual version 
is given in our previous work~\cite{lehecAIHP} together with further applications to functional inequalities. A byproduct of the later version is the following 
additional information on the optimal drift, which we will need later on. 
\begin{lemma}\label{lem_opt}
The supremum in Borell's formula is attained. Moreover the optimal drift $(u_t)$ 
has constant expectation, equal to the barycenter of $\e^\phi$: 
\[ 
\E u_t = \frac{\int_{\R^n} x \e^{\phi(x)} \, \gamma_n (dx) }{\int_{\R^n} \e^{\phi(x)} \, \gamma_n(dx) } , \quad \forall t \in [0,1] . 
\]
\end{lemma} 
\section{Property $(\tau)$}
Property~$(\tau)$ is a concentration inequality introduced by Maurey in~\cite{maurey}.  
It was mainly motivated by some delicate isoperimetric inequality 
for the product of symmetric exponential measures in dimension $n$, 
but here we only need the Gaussian version 
of the inequality, which states as follows. 
\begin{lemma}[Property~$(\tau)$ for the Gaussian measure] 
Suppose $\phi,\psi \colon \R^n \to \R$ satisfy 
\begin{equation}\label{eq_prop_tau_hyp}
\phi (x) + \psi (y) \leq \frac 14 \vert x-y\vert^2, \quad \forall x,y\in \R^n. 
\end{equation} 
Then  
\begin{equation}\label{eq_prop_tau_conc}
\left( \int_{\R^n} \e^\phi \, d\gamma_n \right) 
\left( \int_{\R^n} \e^\psi \, d\gamma_n \right) 
\leq 1. 
\end{equation}
\end{lemma}
\begin{proof}
Let us repeat Borell's proof of the Prékopa-Leindler inequality,
of which the Gaussian property $(\tau)$ is a particular case.  
By monotone convergence it is 
certainly enough to prove the result for functions $\phi$ and $\psi$
which are bounded from above. 
Let $(B_t)$ be a standard Brownian motion on $\R^n$ and let $(u_t)$
and $(v_t)$ be two drifts. Using the hypothesis, 
the Cauchy-Schwarz inequality, and the convexity of the 
Euclidean norm squared we get
\[ 
\begin{split} 
\phi \left( B_1 + \int_0^1 u_t \, dt \right) 
+  \psi \left( B_1 + \int_0^1  v_t \, dt \right) 
 & \leq \frac 14 \left\vert \int_0^1 u_t-v_t \, dt  \right\vert^2\\
 & \leq \frac 14  \int_0^1 \vert u_t-v_t \vert^2  \, dt \\
 & \leq \frac 12 \int_0^1 \vert u_t\vert^2 \, dt + \frac 12 \int_0^1 \vert v_t\vert^2 \,dt. 
\end{split} 
\] 
Taking expectation and then the supremum 
over $(u_t)$ and $(v_t)$ yields the 
desired inequality by Borell's formula. 
\end{proof} 
The constant $1/4$ in the hypothesis~\eqref{eq_prop_tau_hyp} 
is largest possible. This can be seen by taking a linear function for 
$\phi$ and an appropriate quadratic function for 
$\psi$. However, as pointed out in~\cite{AKM}, 
the functional Santal\'o inequality amounts to saying that 
under an additional centering condition, 
this constant can be replaced by $1/2$.
More precisely setting $f(x) = \e^{\phi(x) + \vert x\vert^2/ 2}$ 
in the functional Santal\'o inequality 
(and similarly for $g$) leads to 
the following statement. 
\begin{theorem}[Reformulation of the functional Santal\'o inequality]
Suppose $\phi,\psi \colon \R^n \to \R$ satisfy 
\begin{equation}\label{eq_hypopo}
\phi (x) + \psi (y) \leq \frac 12  \vert x-y\vert^2, \quad \forall x,y\in \R^n 
\end{equation}
and assume additionally that 
\[ 
\int_{\R^n} x \e^{\phi(x) }\, \gamma_n (dx) = 0 . 
\] 
Then the inequality~\eqref{eq_prop_tau_conc} holds true. 
\end{theorem} 
In~\cite{lehecAFST} a direct proof of this inequality is given under the stronger 
assumption that $\phi$ is even, using the Poincar\'e 
inequality for even functions in Gauss space and some symmetrization argument. Property~$(\tau)$ admits 
a dual formulation in terms of transport/entropy inequality. As a result
the functional Santal\'o inequality can be formulated as an improved
transportation inequality for measures satisfying a certain 
centering condition. We refer to~\cite{fathi} 
for the details. This dual version is the approach taken by Courtade, Fathi and 
Mikulincer. Here we stick to the direct version, and we show that the stochastic method of Borell allows to recover the full statement quite easily. 
\section{Stochastic proof of Santal\'o}
We prove the second formulation, in terms of property $(\tau)$. 
We need to leverage the barycenter assumption so as to gain a 
factor $1/2$ in Borell's argument. 
Let $(u_t)$ be the optimal drift in Borell's formula applied to 
the first function $\phi$, and with $(B_t)$ as driving Brownian motion. 
By Lemma~\ref{lem_opt} and the barycenter assumption we have $\E u_t = 0$ 
for all $t\in [0,1]$. 
The main trick, which we borrow from the Courtade, Fathi, Mikulincer 
paper mentioned above, is to use a different Brownian motion for the 
second function, namely the process  
$(\hat B_t)$ given by $\hat B_t = B_1 - B_{1-t}$
for all $t\in[0,1]$. Observe that $(\hat B_t)$ is also 
a standard Brownian motion, and that $\hat B_1 = B_1$. 
Now let $(\hat v_t)$ be any drift with respect the reversed 
process $(\hat B_t)$. Since $\hat B_1 = B_1$ 
we get from the hypothesis (\ref{eq_hypopo})
\begin{equation}\label{eq_df}
\phi \left( B_1 + \int_0^1 u_t \, dt \right) 
+ \psi \left( \hat B_1 + \int_0^1 \hat  v_t \, dt \right) 
\leq \frac 12 \left\vert \int_0^1 u_t \, dt - \int_0^1 \hat v_t \, dt \right\vert^2  
\end{equation} 
Reversing time in the last integral and applying Cauchy-Schwarz we obtain
\begin{equation} \label{eq_tre}
\left\vert \int_0^1 u_t \, dt - \int_0^1 \hat v_t \, dt \right\vert^2 
= \left\vert \int_0^1 u_t \, dt - \int_0^1 \hat v_{1-t} \, dt \right\vert^2 
\leq  \int_0^1 \vert u_t - \hat v_{1-t} \vert^2 \, dt. 
\end{equation} 
Let $(\mathcal F_t)$ and $(\hat{\mathcal F}_{t})$ be the 
natural filtrations of $(B_t)$ and $(\hat B_t)$ respectively. 
From the independence of the Brownian 
increments it is easily seen that for any fixed $t\in[0,1]$ 
the $\sigma$-fields $\mathcal F_t$ and 
$\hat{\mathcal F}_{1-t}$ are independent. 
As a result $u_t$ and $\hat v_{1-t}$ are independent. 
Since the drift $(u_t)$ has expectation $0$ 
for all time we obtain 
\begin{equation}\label{eq_iii}
\E \langle u_t , \hat v_{1-t} \rangle = 
\langle \E u_t , \E \hat v_{1-t}\rangle = 0 , \quad \forall t \in [0,1]  .
\end{equation} 
Combining together (\ref{eq_df}), (\ref{eq_tre}) and (\ref{eq_iii}) yields 
\[ 
\E \left[ \phi \left( B_1 + \int_0^1 u_t \, dt \right) 
+ \psi \left( \hat B_1 + \int_0^1 \hat  v_t \, dt \right) \right]
\leq \E \left[ \frac 12 \int_0^1 \vert u_t\vert^2 \, dt 
+ \frac 12 \int_0^1 \vert \hat v_t\vert^2 \, dt \right] . 
\] 
Recalling that $(u_t)$ is the optimal drift for $\phi$, and 
taking the supremum in $(\hat v_t)$ yields the result.  

\section{Concluding remarks} 

In the aforementioned article \cite{NT}, Nakamura and Tsuji obtain 
the functional Santal\'o inequality as a limit 
case of a certain family of functional inequalities. 
More precisely they establish
an improved form of the reversed hypercontractivity inequality of 
Borell \cite{borell_rev} for even functions. Moreover, they pushed 
this line of work further in \cite{NT2}, in which they show 
that functions satisfying suitable centering conditions satisfy 
stronger forms of the reversed Brascamp-Lieb inequalities 
(of which reversed hypercontractivity is a particular case)
than generic functions. See \cite{barthe_wolff} and the references 
therein for the background on reversed Brascamp-Lieb inequalities.   
Besides, it was shown by E. Milman \cite{milmanGC} 
that these improved reversed Brascamp-Lieb inequalities contain the celebrated Gaussian correlation inequality of Royen \cite{royen} as a special case. 

The Brascamp-Lieb inequalities and their reversed forms are known to be amenable 
to the stochastic approach initiated by Borell, see \cite{lehecBL}. It is 
natural to ask whether the centered versions discovered by Nakamura and 
Tsuji could also be proven this way.  
Unfortunately, we were unable to complete this task, and 
we leave it as an open problem.


\begin{thebibliography}{10}

\bibitem{AKM}
S.~Artstein-Avidan, B.~Klartag, and V.~Milman.
\newblock The {Santal{\'o}} point of a function, and a functional form of the
  {Santal{\'o}} inequality.
\newblock {\em Mathematika}, 51(1-2):33--48, 2004.

\bibitem{ball}
K.~M. Ball.
\newblock {\em Isometric problems in $\ell_p$ and sections of convex sets}.
\newblock PhD thesis, University of Cambridge, 1986.

\bibitem{barthe_wolff}
F.~Barthe and P.~Wolff.
\newblock {\em Positive {Gaussian} kernels also have {Gaussian} minimizers},
  volume 1359 of {\em Mem. Am. Math. Soc.}
\newblock Providence, RI: American Mathematical Society (AMS), 2022.

\bibitem{borell_rev}
C.~Borell.
\newblock Positivity improving operators and hypercontractivity.
\newblock {\em Math. Z.}, 180:225--234, 1982.

\bibitem{borell}
C.~Borell.
\newblock Diffusion equations and geometric inequalities.
\newblock {\em Potential Anal.}, 12(1):49--71, 2000.

\bibitem{CGNT}
D.~Cordero-Erausquin, N.~Gozlan, S.~Nakamura, and H.~Tsuji.
\newblock Duality and heat flow, 2024.

\bibitem{CFM}
T.A. Courtade, M.~Fathi, and D.~Mikulincer.
\newblock Stochastic proof of the sharp symmetrized {Talagrand} inequality.
\newblock {\em C. R., Math., Acad. Sci. Paris}, 362:1779--1784, 2024.

\bibitem{fathi}
M.~Fathi.
\newblock A sharp symmetrized form of {Talagrand}'s transport-entropy
  inequality for the {Gaussian} measure.
\newblock {\em Electron. Commun. Probab.}, 23:9, 2018.
\newblock Id/No 81.

\bibitem{FM}
M.~Fradelizi and M.~Meyer.
\newblock Some functional forms of {Blaschke}-{Santal{\'o}} inequality.
\newblock {\em Math. Z.}, 256(2):379--395, 2007.

\bibitem{lehecAFST}
J.~Lehec.
\newblock The symmetric property {{\((\tau )\)}} for the {Gaussian} measure.
\newblock {\em Ann. Fac. Sci. Toulouse, Math. (6)}, 17(2):357--370, 2008.

\bibitem{lehecCRAS}
J.~Lehec.
\newblock A direct proof of the functional {Santal{\'o}} inequality.
\newblock {\em C. R., Math., Acad. Sci. Paris}, 347(1-2):55--58, 2009.

\bibitem{lehecADM}
J.~Lehec.
\newblock Partitions and functional {Santal{\'o}} inequalities.
\newblock {\em Arch. Math.}, 92(1):89--94, 2009.

\bibitem{lehecAIHP}
J.~Lehec.
\newblock Representation formula for the entropy and functional inequalities.
\newblock {\em Ann. Inst. Henri Poincar{\'e}, Probab. Stat.}, 49(3):885--899,
  2013.

\bibitem{lehecBL}
J.~Lehec.
\newblock Short probabilistic proof of the {Brascamp}-{Lieb} and {Barthe}
  theorems.
\newblock {\em Can. Math. Bull.}, 57(3):585--597, 2014.

\bibitem{maurey}
B.~Maurey.
\newblock Some deviation inequalities.
\newblock {\em Geom. Funct. Anal.}, 1(2):188--197, 1991.

\bibitem{MP}
M.~Meyer and A.~Pajor.
\newblock On the {Blaschke}-{Santal{\'o}} inequality.
\newblock {\em Arch. Math.}, 55(1):82--93, 1990.

\bibitem{milmanGC}
E.~Milman.
\newblock Gaussian {Correlation} via {Inverse} {Brascamp}-{Lieb}.
\newblock Preprint, {arXiv}:2501.11018 [math.{FA}] (2025), 2025.

\bibitem{NT}
S.~Nakamura and H.~Tsuji.
\newblock The functional volume product under heat flow.
\newblock Preprint, {arXiv}:2401.00427 [math.{FA}] (2023), 2023.

\bibitem{NT2}
S.~Nakamura and H.~Tsuji.
\newblock A generalized {Legendre} duality relation and {Gaussian} saturation.
\newblock Preprint, {arXiv}:2409.13611 [math.{FA}] (2024), 2024.

\bibitem{royen}
T.~Royen.
\newblock On improved {Gaussian} correlation inequalities for symmetrical
  {{\(n\)}}-rectangles extended to certain multivariate gamma distributions and
  some further probability inequalities.
\newblock {\em Far East J. Theor. Stat.}, 69(1):1--38, 2025.

\bibitem{S-R}
J.~Saint-Raymond.
\newblock Sur le volume des corps convexes sym{\'e}triques.
\newblock Publ. {Math}. {Univ}. {Pierre} {Marie} {Curie} 46, {S{\'e}min}.
  {Initiation} {Anal}. 20e {Ann{\'e}e}: 1980/1981, {Exp}. {No}. 11, 25 p.
  (1981)., 1981.

\bibitem{santalo}
L.~A. Santaló.
\newblock Un invariante afin para los cuerpos convexos del espacio de n
  dimensiones.
\newblock {\em Portugaliae mathematica}, 8(4):155--161, 1949.

\end{thebibliography}
\end{document}